\documentclass[11pt]{article}
\usepackage{a4wide}
\usepackage{amsmath, amsthm}
\usepackage{amssymb}

\usepackage[hidelinks]{hyperref}
\usepackage{color}
\usepackage[usenames,dvipsnames]{xcolor}

\makeatletter
\newtheorem*{rep@theorem}{\rep@title}

\newenvironment{thmenv}[1]{
\def\rep@title{#1}
\begin{rep@theorem}
}{
\end{rep@theorem}
}
\makeatother

\newcommand{\TT}{\mathcal{T}}

\newcommand{\tendsto}[1]{\xrightarrow[#1]{}}

\renewcommand{\subset}{\subseteq}
\renewcommand{\supset}{\supseteq}

\newif\ifdraft

\newif\ifcolorcomments

\colorcommentstrue

\newcommand{\allowcomments}[4]{
\newcommand{#1}[1]{\ifdraft{\ifcolorcomments{\textcolor{#4}{##1 --#3}}\else{\textsl{##1 \ --#3}}\fi}\else{}\fi}
}

\allowcomments{\comvictor}{VB}{Victor}{blue}
\allowcomments{\comanish}{AG}{Anish}{blue}
\allowcomments{\comdavid}{DS}{David}{Green}
\allowcomments{\comsanju}{SV}{Sanju}{blue}

\font\tenmsy=msbm10 scaled 1200 \font\sevenmsy=msbm7 scaled 1200
\font\fivemsy=msbm5 scaled 1200
\newfam\msyfam
\textfont\msyfam=\tenmsy \scriptfont\msyfam=\sevenmsy
\scriptscriptfont\msyfam=\fivemsy

\newcommand{\R}{{\mathbb R}}
\newcommand{\Rp}{{\mathbb R}^{+}}
\newcommand{\Z}{{\mathbb Z}}
\newcommand{\N}{{\mathbb N}}

\newcommand{\Half}{{\mathbb H}}
\newcommand{\Q}{{\mathbb Q}}

\newcommand{\cH}{{\cal H}}

\newcommand{\p}{\psi}

\newcommand{\I}{{\rm I}}

\newcommand{\ep}{ \varepsilon }

\newtheorem{thDT}{Theorem DT \!\!\!}

\newtheorem{thDTh}{Theorem DT$^{\prime}$ \!\!\!}

\newtheorem{thKT}{Theorem KT \!\!\!}

\newtheorem{thJT}{Theorem JT \!\!\!}

\newtheorem{thkh}{Theorem (Khintchine) \!\!\!\!}

\newtheorem{thjarsch}{Theorem (Jarn\'{\i}k--Schmidt)\!\!}

\newtheorem{lemma}{Lemma}
\newtheorem{theorem}{Theorem}
\newtheorem{corollary}{Corollary}

\theoremstyle{definition}

\theoremstyle{remark}
\newtheorem{rem}{Remark}[section]
\newtheorem{question}{Question}[section]

\newcommand{\dist}{\mathrm{dist}\,}

\newcommand{\Sing}{\mathrm{Sing}}

\newcommand{\cM}{\mathcal{M}}
\newcommand{\cC}{\mathcal{C}}

\newcommand{\be}{\begin{eqnarray*}}
\newcommand{\ee}{\end{eqnarray*}}

\newcommand{\x}{\mathbf{x}}

\newcommand{\pbf}{\mathbf{p}}
\newcommand{\pen}{{\rm pen}}

\newcommand{\bad}{\mathbf{Bad}}

\begin{document}

\title{\huge Diophantine approximation in Kleinian groups: singular, extremal, and bad limit points}

\author{Victor Beresnevich\footnote{Research partly supported by EPSRC grant EP/J018260/1}
\\ {\small\sc (York) } \and Anish Ghosh\footnote{Research partly supported by a UGC grant} \\ {\small\sc (TIFR)} \and David Simmons\footnote{Research supported by EPSRC grant EP/J018260/1 } \\ {\small\sc (York)} \and Sanju Velani\footnote{Research partly supported by EPSRC grant EP/J018260/1 } \\ {\small\sc (York)} }

\date{ ~ \\ {\em Dedicated to Paddy Patterson } 
}

\maketitle

\begin{abstract}

The overall aim of this note is to initiate a ``manifold'' theory for metric Diophantine approximation on the limit sets of Kleinian groups. We investigate the notions of singular and extremal limit points within the geometrically finite Kleinian group framework. Also, we consider the natural analogue of Davenport's problem regarding badly approximable limit points in a given subset of the limit set. Beyond extremality, we discuss potential Khintchine-type statements for subsets of the limit set. These can be interpreted as the conjectural ``manifold'' strengthening of Sullivan's logarithmic law for geodesics.

\end{abstract}

\bigskip

%
%
%

\section{The general setup and main problems \label{gensetup}}

The classical results of Diophantine approximation,\hspace{2pt} in particular those from the one-dimensional theory, have natural counterparts and extensions in the hyperbolic space setting. In this setting, instead of approximating real numbers by rationals, one approximates the limit points of a fixed Kleinian group $G$ by points in the orbit (under the group) of a distinguished limit point $y$. Beardon $\&$ Maskit \cite{BeaMas} have shown that the geometry of the group is reflected in the approximation properties of points in the limit set.

Unless stated otherwise, in what follows $G$ denotes a nonelementary, geometrically finite Kleinian group acting on the unit ball model $(B^{d+1}, \rho)$ of $(d+1)$--dimensional hyperbolic space with metric $ \rho$ derived from the differential $ d \rho = 2 | d \x | /(1-|\x |^2 ) $. Thus, $G$ is a discrete subgroup of $ \mbox{M\"ob}(B^{d+1})$, the group of orientation-preserving M\"obius transformations of the unit ball $B^{d+1}$. By assumption, there is some finite-sided convex fundamental polyhedron for the action of $G$ on $B^{d+1}$. Since $G$ is nonelementary, the limit set $\Lambda$ of $G$ (the set of limit points in the unit sphere $S^d$ of any orbit of $G$ in $B^{d+1}$) is uncountable. The group $G$ is said to be \emph{of the first kind}\footnote{A geometrically finite group of the first kind is also called a \emph{lattice}.} if $\Lambda = S^d $ and \emph{of the second kind} otherwise. Let $ \delta $ denote the Hausdorff dimension of $\Lambda$. Trivially, if $G$ is of the first kind then we have $\delta := \dim \Lambda = d $. In general, it is well known that $\delta $ is equal to the exponent of convergence of the group \cite{PaddyLim, SullEnt}. For each element $g \in G$ we shall use the notation $L_g := |g^\prime(0)|^{-1}$, where $|g^\prime(0)| = 1-|g(0)|^2 $ is the (Euclidean) conformal dilation of $g$ at the origin. It can be verified that $L_g \leq e^{\rho(0,g(0))} \leq 4 L_g$. With this setup and notation in mind, we are in the position to state three fundamental results originating from Patterson's pioneering paper \cite{Paddyrs}. In short, they represent natural generalisations to the hyperbolic space setting of the classical theorems of Dirichlet, Khintchine, and Jarn\'{\i}k in the theory of Diophantine approximation. In view of this, they naturally motivate our Kleinian group investigation into singular, extremal, and badly approximable points in $\Lambda$ and its subsets.

\subsection{A Dirichlet-type theorem and singular subsets of $\Lambda$} \label{DTTsec}

The following two Dirichlet-type theorems were first established by Patterson \cite[Section 7: Theorems 1 \& 2]{Paddyrs} for finitely generated Fuchsian groups, i.e. Kleinian groups acting on the unit disc model of $2$--dimensional hyperbolic space. Recall that in this $d=1$ case, the class of finitely generated groups coincides with the class of geometrically finite groups.

\begin{thDT}
Let $G$ be a nonelementary, geometrically finite Kleinian group containing parabolic elements and let $P$ be a complete set of inequivalent parabolic fixed points of $G$. Then there is a constant $c> 0 $ with the following property: for each $\xi \in \Lambda$, $ N > 1$, there exist $p \in P$, $g \in G$ so that
\[
| \xi - g(p) | \le \frac{c}{\sqrt{L_g N}}
\qquad and \qquad L_g \le N \, .
\]
\end{thDT}

\noindent As pointed out in \cite{Squad}, the $d=1$ proof of Patterson can be easily generalised to higher dimensions when the ranks\footnote{The stabiliser $ G_p =\{ g \in G \, : \, g(p) = p \} $ of a parabolic fixed point $p$ is an infinite group which contains a free abelian subgroup of finite index. The \emph{rank} of $p$ is defined to be the number $k\in [1,d]$ such that this subgroup is isomorphic to $\Z^k$.} of the parabolic fixed points are all maximal; i.e. when $\mathrm{rank}(p)= d $ for all $p \in P$. Without this rank assumption, the theorem is proved in \cite[Theorem 1]{Strap}. We now consider the case where the geometrically finite group $G$ has no parabolic elements; i.e. where $G$ is convex cocompact.

\begin{thDTh}
Let $G$ be a nonelementary, geometrically finite Kleinian group without parabolic elements and let $\{\eta,\eta'\}$ be the pair of fixed points of a hyperbolic element of $G$. Then there is a constant $c> 0 $ with the following property: for all $\xi \in \Lambda$, $ N > 1$, there exist $y \in \{\eta,\eta'\}$, $g \in G$ so that
\[
| \xi - g(y) | \le \frac{c}{ N} \qquad and \qquad L_g \le N \, .
\]
\end{thDTh}

\noindent Patterson's $d=1$ proof of the above theorem easily generalises to higher dimensions.

\medskip

When interpreted on the upper half-plane $\Half^2$ and applied to the modular group $SL(2,\Z)$, it is easily verified that Theorem DT reduces to the  $d=1$ case of Dirichlet's Theorem. Recall that Dirichlet's Theorem states that for all $\x =(x_1, \ldots,x_d) \in \R^d$, $ N \in \N$, there exist $\pbf = (p_1, \ldots,p_d) \in \Z^d$, $ q \in \N $ so that
\[
\max_{1 \le i \le d} \Big | x_i - \frac{p_i}{q} \Big| \le \frac{1}{ q N^{\frac1d}} \qquad {\rm and } \qquad q \le N \, .
\]
Staying within the classical setup, a vector $\x \in \R^d$ is said to be \emph{singular} if for every $ \ep > 0 $ there exists $ N_0$ with the following property: for each $ N \ge N_0$, there exist $\pbf \in \Z^d$, $ q \in \N $ so that
\begin{equation}\label{singclassical}
\max_{1 \le i \le d} \Big | x_i - \frac{p_i}{q} \Big| < \frac{\ep}{ q N^{\frac1d}} \qquad {\rm and } \qquad q < N \, .
\end{equation}
In short, $ \x $ is singular if Dirichlet's Theorem can be ``improved'' by an arbitrarily small constant factor $\ep>0$. It is not difficult to see that the set $\Sing (d)$ of singular vectors contains every rational hyperplane in $\R^d$ and thus its Hausdorff dimension is between $d-1$ and $d$. In the case $d=1$, a nifty argument (which we shall utilise) due to Khintchine \cite{kh} shows that a real number is singular if and only if it is rational; that is, $\Sing (1) = \Q$. Davenport $\&$ Schmidt \cite{DavSch} in the seventies showed that $ \Sing (d) $ is a set of $d$-dimensional Lebesgue measure zero. Recently, Cheung $\&$ Chevallier \cite{ChCh}, building on the spectacular $d=2$ work of Cheung \cite{Ch}, have shown that $ \Sing (d) $ has Hausdorff dimension $ \frac{d^2}{d+1}$.

Motivated by the above classical ``singular'' theory we introduce the notion of singular limit points within the hyperbolic space setup. Let $G$ be a Kleinian group and let $Y$ be a complete set $P$ of inequivalent parabolic fixed points of $G$ if the group has parabolic elements; otherwise let $Y$ be the pair $\{\eta,\eta'\}$ of fixed points of a hyperbolic element of $G$. A point $\xi \in \Lambda $ is said to be \emph{singular} if for every $ \ep > 0 $ there exists $ N_0$ with the following property: for each $ N \ge N_0$, there exist $y \in Y$, $g \in G$ so that
\begin{equation}\label{singhyper}
| \xi - g(y) | \ < \ \left\{
\begin{array}{ll}
\frac{\ep}{\sqrt{L_g N}} \quad & {\rm if}
\;\;\; Y = P \; \\ [4ex]
\frac{\ep}{ N } \quad & {\rm if} \;\;\; Y = \{\eta,\eta'\} \;
\end{array}
\right.
\qquad {\rm and } \qquad
L_g < N \ . \end{equation}

\noindent Our first result shows that the hyperbolic ``singular'' theory is not as rich as the higher-dimensional classical theory in $\R^d$. Indeed, irrespective of the dimension the hyperbolic space, it is in line with the one-dimensional classical theory.

\begin{theorem} \label{singthm}
Let $G$ be a nonelementary, geometrically finite Kleinian group, and let $Y$ be as above. Then a point $ \xi \in \Lambda $ is singular if and only if $\xi \in G(Y) := \{ g(y) : g \in G, y \in Y \}$.
\end{theorem}

\vspace{1ex}

\begin{rem}
In the case where $G$ is convex cocompact, the set of singular limit points is dependent on the choice of $Y$; i.e. on the chosen pair $\{\eta,\eta'\}$ of hyperbolic fixed points of $G$. If $G$ has parabolic elements, the set of singular limit points is precisely the set of parabolic fixed points of $G$. Dynamically, the set corresponds to geodesics on the associated hyperbolic manifold $ \cH = B^{d+1} /G $  that travel straight into the ``throat'' of a cuspidal end -- see \S\ref{dreamGs}.
\end{rem}

\vspace{1ex}

\begin{rem}
The parabolic fixed points of the modular group are the rationals together with the point at infinity. Thus, Theorem \ref{singthm} when interpreted on $\Half^2$ and applied to $SL(2,\Z)$ precisely coincides with the $d=1$ classical results.
\end{rem}

\vspace{1ex}

\begin{rem}
The proof of Theorem \ref{singthm} is pretty straightforward and relies on the disjointness property of horoballs based at the parabolic or hyperbolic fixed points associated with the set $G(Y)$ -- see \S\ref{prrofthm1} for the details.
\end{rem}


\subsection{A Khintchine-type theorem and extremal subsets of $\Lambda$} \label{KTTsec}
Let $G$ be a nonelementary, geometrically finite Kleinian group $G$ and let $y$ be a parabolic fixed point of $G$ if the group has parabolic elements and a hyperbolic fixed point otherwise. The Dirichlet-type theorems of \S\ref{DTTsec} together with natural ``decoupling'' results (see for example \cite[Proposition 2.3]{Strap} and \cite[Proposition 2]{Squad}) imply the following statement for any nonelementary, geometrically finite Kleinian group $G$: {\em for each point $\xi \in \Lambda$ which is not a parabolic fixed point there exist infinitely many $g \in G$ such that }
\begin{equation} \label{KGDTio}
| \xi - g(y) | < \frac{c}{L_g} \, .
\end{equation}
Here, $c$ is a positive group constant. It is easy to see that if $G$ has only one equivalence class of parabolic fixed points then we can take $\xi$ to be any limit point. In any case, the statement describes to what extent any (non-parabolic) limit point $\xi$ may be approximated by the orbit of the distinguished point $y$; namely that every non-parabolic limit point can be approximated by orbit points $g(y)$ with ``rate'' of approximation given by $c/L_g$ -- the right-hand side of inequality \eqref{KGDTio} determines the ``rate'' or ``error'' of approximation. It is natural to broaden the discussion to include general approximating functions. More precisely,
let $\psi : \Rp \to \Rp:= [0,\infty) $ be a decreasing function and let
\[
W_{y}(\psi)= W_{y}(\psi,G)
:=\big\{ \xi \in \Lambda: | \xi - g(y) | < \psi(L_g) \, \mbox{ \ for \ i.m.
$g \in G$}\big\} .
\]
\noindent This is the set of points in the limit set which are ``close'' to infinitely many (``i.m.'') images of the ``distinguished'' point $y$. The degree of ``closeness'' is of course governed by the approximating function $\psi$. As above, $y$ is taken to be a parabolic fixed point of $G$ if the group has parabolic elements and a hyperbolic fixed point of $G$ otherwise. A natural problem is to determine the ``size'' of the set $ W_{y}(\psi) $ in terms of the Patterson measure $m$ -- a nonatomic, $\delta$-conformal probability measure $m$ supported on $\Lambda$. For groups of the first kind, since $\delta:= \dim \Lambda = d $, the Patterson measure is simply normalised $d$-dimensional Lebesgue measure on the unit sphere $ S^d$. The following Khintchine-type theorem was first established by Patterson \cite[Section 9]{Paddyrs} for finitely generated Fuchsian groups of the first kind. For convenience, let
\[
w(y) := \left\{
\begin{array}{ll}
2\delta-{\rm rank}(y) & \mbox{ if \;\;\; $y$ is parabolic},\\ [2ex]
\delta & \mbox{ if \;\;\; $y$ is hyperbolic}.
\end{array}
\right.
\]
\vspace{3ex}

\begin{thKT}
Let $G$ be a nonelementary, geometrically finite Kleinian group and let $y$ be a parabolic fixed point of $G$, if there are any, and a hyperbolic fixed point otherwise. Then

\[
m( W_{y}(\psi) ) = \left\{
\begin{array}{ll}
0 & {\rm \ if} \;\;\;
\sum_{r=1}^\infty \; \psi\left(r\right)^{w(y)}
\;\; r^{w(y)- 1} <\infty\; ,\\ [3ex]
1 & {\rm \ if} \;\;\; \sum_{r=1}^\infty \;
\psi\left(r\right)^{w(y)} \;\; r^{w(y)- 1} =\infty \; .
\end{array}
\right.
\]
\label{GF1}
\end{thKT}

\vspace{1ex}

\begin{rem} \label{like} In terms of this note, there are two special cases of the above theorem that are of particular interest to us.
\begin{itemize}
\item[(i)] For $\ep \ge 0 $, let $ \psi_{\ep} : r \to r^{-1} (\log r)^{- \frac{ 1+ \ep}{w(y)}} $. Then it follows that
\[
m( W_{y}(\psi_{\ep}) ) = \left\{
\begin{array}{ll}
0 & {\rm \ if} \;\;\;
\ep > 0 \; ,\\ [1ex]
1 & {\rm \ if} \;\;\; \ep = 0 \; .
\end{array}
\right.
\]
This statement has a well-known dynamical interpretation in terms of the ``rate'' of excursions by geodesics into a cuspidal end of the associated hyperbolic manifold $ \cH =  B^{d+1} /G $; namely Sullivan's logarithm law for geodesics \cite{paris, Strap, SDS}. We shall return to this in \S\ref{dreamGs}.

\item[(ii)] For $ \tau \ge 1 $, consider the function $ \psi : r \mapsto r^{-\tau} $ and write $ W_{y}(\tau) $ for $ W_{y}(\psi)$. Then it follows that
\[
m( W_{y}(\tau) ) = 0 \qquad {\rm \ if} \;\;\; \tau > 1 \; .
\
\]
\label{GF2}
The fact that $ m( W_{y}(\tau) ) = 1$ for $\tau = 1$ can be easily deduced from the statement associated with inequality \eqref{KGDTio}  and the fact that $  m(W_{y}(c\psi)) =  m(W_{y}(\psi))  $ for any constant  $c> 0$ \cite[Lemma 4.6]{Strap}  - we do not need the full power of the divergence case of Theorem~KT.
\end{itemize}
\end{rem}

\noindent Without assuming that $G$ is of the first kind, Theorem KT is essentially established in \cite{StrapCoCompact} if $y$ is a hyperbolic fixed point of $G$ and in \cite{Strap} if $y$ is a parabolic fixed point of $G$. We say essentially, since in both \cite{StrapCoCompact} and \cite{Strap} an extra regularity condition on the approximating function $\psi$ is assumed. The theorem as stated above, without any regularity condition on $\psi$ beyond monotonicity, is established in \cite[Section 10.3: Theorems 5 $\&$ 9]{memoirs} and is the perfect Kleinian group analogue of Khintchine's Theorem in the classical theory of metric Diophantine approximation. Indeed, when interpreted on the upper half-plane $\Half^2$ and applied to the modular group $SL(2,\Z)$, it is easily verified that Theorem KT reduces to the $d=1$ case of Khintchine's Theorem. In what follows, $W(d,\psi)$ denotes the set of simultaneously $\psi$-well approximable points in the unit cube $\I^d:= [0,1]^d$; that is,
\[
W(d,\psi) := \left\{ \x\in \I^d: \max_{1 \le i \le d} \Big | x_i - \frac{p_i}{q} \Big| \le \psi(q) \, \mbox{ \ for \ i.m.
$(\pbf,q) \in \Z^d \times \N$} \right\} \, .
\]
\vspace{1ex}

\begin{thkh} Let $m_d$ be $d$-dimensional Lebesgue measure. Then
\[
m_{d} (W(d,\psi)) =\left\{
\begin{array}{ll}
0 & {\rm if} \;\;\; \sum_{r=1}^{\infty} \; (\psi(r) \, r )^d <\infty\;
,\\ [3ex]
1 & {\rm if} \;\;\; \sum_{r=1}^{\infty} \; (\psi(r) \, r )^d =\infty \; .
\end{array}\right.
\]
\end{thkh}

\vspace{2ex}

Staying within the classical setup, we briefly turn to the manifold theory. In short, Diophantine approximation on manifolds is the study of the Diophantine properties of points in $\R^d$ whose coordinates are constrained by
(differentiable) functional relations, or equivalently points which are known to be members of a submanifold $\cM \subseteq \R^d$. Actually, there is no harm in restricting our attention to submanifolds $\cM \subseteq \I^d $, and the specific aspect of the manifold theory that we will be concerned with is that of describing the measure of $ \cM \cap W(d,\psi) $ (with respect to the Lebesgue measure on $\cM$). The fact that the points of interest $\x \in \I^d$ are constrained by functional relations, or in other words that they are required to be members of a fixed manifold $\cM$, introduces major difficulties in attempting to analyse the measure-theoretic structure of $ \cM \cap W(d,\psi) $. This is true even for seemingly simple curves such as the unit circle or the parabola.

The goal is to obtain a Khintchine-type theorem that describes the Lebesgue measure of the set of simultaneously $\psi$--approximable points lying on any given manifold. Notice that if the dimension $k$ of the manifold ${\cal M}$ is strictly less than $d$ then $ m_d(\cM \cap W(d,\psi)) = 0 $ irrespective of the approximating function $\psi$. Thus, in attempting to develop a general Lebesgue theory for $ \cM \cap W(d,\psi)$ it is natural to use the normalised $k$-dimensional Lebesgue measure on $\cM$. This will be denoted by $ | \ \cdot \ |_{\cM} $. In order to make any reasonable progress with developing a general theory, we insist that the manifolds $\cM$ under consideration are {\em nondegenerate manifolds}. Essentially, these are smooth submanifolds of $\R^d$ which are sufficiently curved so as to deviate from any hyperplane. For a formal definition and indeed a more in-depth overview of the manifold theory, we refer the reader to \cite[Section 6]{durham} and the references within. In terms of examples, any connected analytic manifold not contained in any hyperplane of $\R^d$ is nondegenerate. Also, a planar curve ${\cal C}$ is nondegenerate if the set of points on ${\cal C}$ at which the curvature vanishes is a set of one-dimensional Lebesgue measure zero.

The claim is that the notion of nondegeneracy is the right criterion for a manifold $\cM$ to be ``sufficiently'' curved in order to obtain a Khintchine-type theorem (both convergence and divergence cases) for $\cM \cap W(d,\psi)$.

\begin{thmenv}{Conjecture 1 (The Dream Theorem)}
Let $\cM$ be a nondegenerate submanifold of $\R^d$. Then
\begin{equation}\label{vb5}
| \cM\cap W(d,\psi) |_{\cM}
=\left\{
\begin{array}{ll}
0 & \textup{if} \;\;\; \sum_{r=1}^{\infty} \; (\psi(r) \, r )^d <\infty\;
,\\ [4ex]
1 & \textup{if} \;\;\; \sum_{r=1}^{\infty} \; (\psi(r) \, r )^d =\infty \; .
\end{array}\right. 
\end{equation}
\end{thmenv}

\noindent We now describe various ``general'' contributions towards the Dream Theorem. Let us write $ W(d,\tau) $ for $ W(d,\psi)$ when considering functions $\psi$ of the shape $ \psi( r ) = r^{-\tau} $.

\begin{itemize}

\item { \em Extremal manifolds.}\index{Extremal manifolds} A submanifold $\cM$ of $\R^d$ is called {\em extremal} if
\[
\left| \cM\cap W(d, \tau) \right|_{\cM} =0
\qquad \forall \;\; \tau > \textstyle{\frac{d+1}{d}} \,.
\]
Note that Dirichlet's theorem implies that $ W(d, \textstyle{\frac{d+1}{d}} ) = \I^d $ and so it trivially follows that $ \cM\cap W(d, \textstyle{\frac{d+1}{d}} ) = \cM$. In their pioneering work \cite{KM98}, Kleinbock $\& $ Margulis proved that any nondegenerate submanifold $\cM$ of $\R^d$ is extremal. It is easy to see that this implies the convergence case of the Dream Theorem for functions $ \psi : r \mapsto r^{-\tau} $. It is worth mentioning that Kleinbock $\& $ Margulis established a stronger (multiplicative) form of extremality that settled the Baker--Sprind\v zuk Conjecture from the eighties.

\item {\em Planar curves.} The Dream Theorem is true when $d=2$; that is, when $\cM$ is a nondegenerate planar curve. The convergence case of \eqref{vb5} for planar curves was established in \cite{Vaughan-Velani-2007} and subsequently strengthened in \cite{BZ}. The divergence case of \eqref{vb5} for planar curves was established in \cite{Beresnevich-Dickinson-Velani-07:MR2373145}.

\item {\em Beyond planar curves.} The divergence case of the Dream Theorem is true for analytic nondegenerate submanifolds of $\R^d$ \cite{B12}. In current work \cite{BVVZdiv} being written up,  the divergence case of \eqref{vb5} will be shown to be true for nondegenerate curves, as well as manifolds that can be ``fibred'' into such curves \cite{BVVZdiv}. The latter includes $C^{\infty}$ nondegenerate submanifolds of $\R^d$ which are not necessarily analytic. The convergence case of the Dream Theorem is true for a large class of nondegenerate submanifolds of $\R^d$ with dimension $k$ satisfying $k(k+3)/2 > d$, and this class includes ``most'' manifolds when $k(k+1)/2 \geq d$ \cite{Simmons}.  The work in  \cite{Simmons} builds upon  the approach taken in \cite{BVVZcon} in which the convergence case is shown to be true for a large subclass of nondegenerate submanifolds with  $k>(d+1)/2$.

\end{itemize}

\noindent The upshot of the above is that the Dream Theorem actually holds for a fairly generic class of nondegenerate submanifolds $\cM$ of $\R^d$ apart from the case of convergence when $d \ge 3 $ and $k(k+1)/2 < d$.

\vspace*{1ex}

\begin{rem} \label{friendlydecay}
In \cite{Kleinbock-Lindenstrauss-Weiss-04:MR2134453}, Kleinbock, Lindenstrauss, $\&$ Weiss made an emphatic generalisation of the ``extremal'' work of \cite{KM98} to subsets $K$ of $\R^d$ supporting so-called {\em friendly } measures. Within the context of this paper, it suffices to say that friendly measures form a large and natural class of measures on $\R^d$ which includes Riemannian measures supported on nondegenerate manifolds, fractal measures supported on self-similar sets satisfying the open set condition (e.g. regular Cantor sets, the Koch snowflake, the Sierpi\'nski gasket), and conformal (Patterson) measures supported on the limit sets of geometrically finite Kleinian groups, as long as they are not contained in any hyperplane. These facts are proven in \cite[Theorem 2.3]{Kleinbock-Lindenstrauss-Weiss-04:MR2134453} and \cite[Theorem 1.9]{DFSU_GE2}, respectively. Recently, the concept of friendly measures has been generalised even further to the notion of \emph{quasi-decaying} measures, see \cite{DFSU_GE1,DFSU_GE2}.
\end{rem}

In view of the recent progress within the classical manifold theory, it would be highly desirable to obtain an analogous theory within the hyperbolic space setup. With this in mind as the ultimate goal, let $K$ be a subset of the limit set $\Lambda$ which supports a nonatomic probability measure $\mu$. Then $K$ will play the role of the manifold and $\mu$ the role of the Lebesgue measure on the manifold. In this note, we develop a reasonably complete extremal theory for Kleinian groups. In view of \eqref{KGDTio}, it is natural to say that a subset $K \subseteq \Lambda$ is $\mu$-extremal if
\[
\mu( K \cap W_{y}(\tau) ) = 0 \qquad {\forall} \;\; \tau > 1 \; .
\label{GF3}
\]
Note that $\Lambda$ is $m$-extremal where $m$ is the Patterson measure --- see Remark \ref{friendlydecay}. To have any hope of developing a general extremal theory for the subsets $K$ we impose the following ``decaying'' condition on the measure $\mu$. Given $\alpha > 0$, the measure $\mu$ supported on $K$ is said to be {\em weakly absolutely $\alpha$-decaying}  if there exist strictly positive constants $ C, r_0 $ such that for all
$\ep > 0$ we have
\[
\mu\big(B(x,\ep r) \big) \ \leq \ C \, \ep^{\alpha} \mu\big(B(x,r)\big) \hspace{7mm} \forall \ x \in K
\hspace{5mm} \forall \ r < r_0 \ .
\]
For sets supporting such measures, we are able to prove the following result.

\begin{theorem} \label{mainext}
Let $G$ be a nonelementary, geometrically finite Kleinian group and let $y$ be a parabolic fixed point of $G$, if there are any, and a hyperbolic fixed point otherwise. Fix $\alpha > 0$, and let $K$ be a compact subset of $\Lambda$ equipped with a weakly absolutely $\alpha$-decaying measure $\mu$. Then
\begin{equation} \label{ineqmain2}
\mu( K \cap W_{y}(\psi) ) = 0\hspace{6mm} {\rm if \ } \hspace{6mm} \sum_{r=1}^{\infty}
r^{\alpha -1 } \p(r)^{\alpha} \ < \ \infty \ .
\end{equation}
\end{theorem}

\noindent

\vspace{1ex}

%

\begin{rem} \label{absdecay}
It is easily verified that if a measure  $\mu$ is
absolutely $\alpha$-decaying as defined in \cite{PVfriends} then it is weakly absolutely $\alpha$-decaying.  Also it is worth pointing out that although the Lebesgue measure $ | \ . \ |_{\cM} $ on a nondegenerate manifold $\cM$
is not necessarily absolutely $\alpha$-decaying,   it is  weakly absolutely $\alpha$-decaying.
\end{rem}

\vspace{1ex}

Observe that if we write $\psi_\tau(r) = r^{-\tau}$, then

\[
\sum_{r=1}^{\infty}
r^{\alpha -1 } \p_\tau(r)^{\alpha} = \sum_{r=1}^{\infty}
r^{\alpha(1- \tau) -1 } \ < \ \infty \qquad {\forall} \; \tau > 1 \;\; \forall \; \alpha > 0.
\]
Hence the following statement is a trivial consequence of Theorem \ref{mainext}.

\begin{corollary} \label{main2} Let $G$ be a nonelementary, geometrically finite Kleinian group and let $y$ be a parabolic fixed point of $G$, if there are any, and a hyperbolic fixed point otherwise. Let $K$ be a compact subset of $\Lambda$ equipped with a weakly absolutely decaying measure $\mu$. Then $K$ is $\mu$-extremal.
\end{corollary}

\medskip

The deeper and more subtle analogue of the Dream Theorem for Kleinian groups is the subject of \S\ref{dreamG}.



\subsection{A Jarn\'{\i}k-type theorem and ``Bad'' subsets of $\Lambda$}

We motivate the contents of this section by returning to inequality \eqref{KGDTio} and asking: can the group constant $c>0$ be made arbitrarily small? In other words, if we let $ \psi_{\ep} : r \mapsto \ep \, r^{-1}$, then do we have $ W_{y}(\psi_{\ep}) \supseteq \Lambda\setminus G(P)$ for all $\ep > 0$? It follows immediately from Theorem KT that $ m(W_{y}(\psi_{\ep})) = 1 = m( \Lambda\setminus G(P))$. Thus, the set of exceptions to the above inclusions, i.e. the set
\[
\bad_y := \left\{ \xi \in \Lambda : \ \exists \ \ c(\xi) > 0 {\rm \ such \ that \ } | \xi - g(y) | > c(\xi)/ L_g \ \ \forall \ \
g \in G \right\} \, ,
\]
is of $m$-measure zero. Nevertheless, the answer to the above question is emphatically no since the exceptional set of ``badly approximable'' limit points has full Hausdorff dimension. The following Jarn\'{\i}k-type theorem was first established by Patterson \cite[Section 10]{Paddyrs} for finitely generated Fuchsian groups of the first kind. As usual, $y$ is taken to be a parabolic fixed point of $G$ if the group has parabolic elements and a hyperbolic fixed point of $G$ otherwise.

\vspace{3ex}

\begin{thJT}
Let $G$ be a nonelementary, geometrically finite Kleinian group and let $y$ be a parabolic fixed point of $G$, if there are any, and a hyperbolic fixed point otherwise. Then
\[
\dim \bad_y = \dim \Lambda \, .
\]
\end{thJT}

\begin{rem} \label{bdedbad}
When $G$ has parabolic elements, a stronger version of the above theorem is known: if $P$ is a complete set of inequivalent parabolic points, then
\[
\dim\bigcap_{p\in P} \bad_p = \dim \Lambda \, .
\]
This stronger theorem has a well-known dynamical interpretation; namely that the set of bounded geodesics on the associated hyperbolic manifold $ \cH =  B^{d+1} /G $ is of full dimension.
\end{rem}

\noindent Without assuming that $G$ is of the first kind, Theorem JT is established in \cite{StrapCoCompact} if $y$ is a hyperbolic fixed point of $G$ and in \cite{StrapbadGF} if $y$ is a parabolic fixed point of $G$. When interpreted on the upper half-plane $\Half^2$ and applied to the modular group $SL(2,\Z)$, it is easily verified that Theorem JT reduces to the $d=1$ case of the Jarn\'{\i}k--Schmidt Theorem on the size of the classical set $\bad(d) $ of simultaneously badly approximable numbers. Recall that $\bad(d) $ is the set of $\x\in \R^d $ such that there exists a positive constant $c(\x) >0 $ so that

\[
\max_{1 \le i \le d} \Big | x_i - \frac{p_i}{q} \Big| \ \ge \ c(\x) \,   q^{-\frac{d+1}{d}} \quad \forall \ \
(\pbf,q) \in \Z^d \times \N \, .
\]
\vspace*{2ex}

\begin{thjarsch}
For $d \ge 1$, we have that
$
\dim \bad(d) = d . \,
$
\end{thjarsch}

The $d=1$ case is due to Jarn\'{\i}k (1928) while the general statement is due to Schmidt (1969).  Indeed, Schmidt showed that the set $\bad(d)$ satisfies a stronger ``winning'' property associated with a certain game that now bears his name. Staying within the classical setting, we turn to the badly approximable version of the manifold theory described in \S\ref{KTTsec}. The following statement is the natural analogue of Conjecture~1.

\begin{thmenv}{Conjecture 2} Let $\cM$ be a nondegenerate submanifold of $\R^d$. Then
\begin{equation*}\label{vb23}
\dim ( \cM\cap \bad(d) ) = \dim \cM \, .
\end{equation*}
\end{thmenv}

\noindent We now describe various ``general'' contributions towards Conjecture 2.

\begin{itemize}

\item {\em Planar curves.} The conjecture is true when $d=2$; that is, when $\cM$ is a nondegenerate planar curve. This was established in \cite{Badziahin-Velani-Dav} and independently in \cite{BadM} and provides a solution to a problem of Davenport dating back to the sixties concerning the existence of badly approximable pairs on the parabola. The stronger ``winning'' property has subsequently been established in \cite{An-Beresnevich-Velani}.

\item {\em Beyond planar curves.} The conjecture is true for analytic nondegenerate submanifolds of $\R^d$ \cite{BadM}. The condition of analyticity can be omitted in the case where the submanifold $\cM\subseteq\R^d$ is a curve. Indeed, establishing the result for curves is very much at the heart of the approach in \cite{BadM}.
\end{itemize}

\noindent For a more in-depth overview of the badly approximable manifold theory, we refer the reader to \cite[Section 7]{durham} and the references within.

\vspace*{2ex}

Motivated by the above  (badly approximable) manifold theory in $\R^d$, we aim to develop an analogous theory within the hyperbolic space setting. Thus, as in \S\ref{KTTsec}, let $K$ be a subset of the limit set $\Lambda$ which supports a nonatomic probability measure $\mu$. We would like to conclude that $ K \cap \bad_y $ is of full dimension for a general class of sets $K$. With this in mind, we impose the condition that the measure $\mu$ supported on $K$ is {\em Ahlfors $\delta$-regular} for some $\delta > 0$; that is, that there exist strictly positive constants $ C, r_0 $ such that
\[
C^{-1} \, r^{\delta}\ \le \ \mu\big(B(x, r) \big) \ \le \ C \, r^{\delta} \hspace{7mm} \forall \ x \in K
\hspace{5mm} \forall \ r < r_0 \ .
\]
\noindent  Sets supporting such measures are referred to as Ahlfors $\delta$-regular and it is a well known fact that
 $$
 \dim K =  \delta  \, .
 $$
 For Ahlfors $\delta$-regular  subsets of the limit set  we are able to prove the following result.

\begin{theorem} \label{mainbad}
Let $G$ be a nonelementary, geometrically finite Kleinian group and let $y$ be a parabolic fixed point of $G$, if there are any, and a hyperbolic fixed point otherwise. Let $K$ be a compact, Ahlfors $\delta$-regular subset of $\Lambda$. Then
\begin{equation} \label{ineqmain3}
\dim \left( K \cap \bad_y \right) = \dim K \ . \end{equation}
\end{theorem}

\noindent

\vspace{1ex}

\begin{rem} \label{reg}
As we shall see in \S\ref{proofmainbad}, the above theorem, although new, is a consequence of combining various recent results. \end{rem}

 \begin{rem}  The methods used to establish the main results in this paper (namely Theorems \ref{singthm}, \ref{mainext}, $\&$ \ref{mainbad})  can almost certainly  be adapted to prove analogous statements for rational maps.  Within the context of Theorem \ref{singthm} $\&$ \ref{mainext},  Sky Brewer is currently developing a general framework that naturally incorporates both the Kleinian group and rational map setup. As we shall see, the proof of Theorem \ref{mainbad} already makes use of  a  powerful and general framework for investigating badly approximable sets.
 \end{rem}


\section{Proof of Theorems \ref{singthm} $\&$ \ref{mainext} }

We begin with a preliminary section in which we provide the reader with necessary concepts and results required in the proof of Theorems \ref{singthm} $\&$ \ref{mainext}. It also enables us to geometrically restate the theorems in terms of horoballs from which it is relatively straightforward to derive their dynamical interpretations in terms of geodesic excursions on the associated hyperbolic manifolds. A {\em horoball} $H_{\xi} $ based at $\xi \in  S^d$ is an open $(d+1)$-dimensional Euclidean ball contained in $ B^{d+1}$ such that its boundary $ \partial H_{\xi} $ is tangent to $ S^d $ at the point $ \xi $. For each $\xi \in S^d $, let $s_{\xi}$ be the ray in $B^{d+1}$ joining the origin to $ \xi$. The {\sl top} $ \sigma_{\xi} := s_{\xi} \cap \partial H_{\xi} $ of a horoball $ H_{\xi} $ such that $0 \notin H_\xi$ is the point on $ \partial H_{\xi} $ closest to the origin.

\medskip

\noindent{\sl Notation.} The symbols $ \ll $ and $ \gg $ will be used to indicate an inequality with an implied unspecified positive multiplicative constant factor. If $a \ll b $ and $ a \gg b $ we write $ a \asymp b$, and say that the quantities $a$ and $b$ are \emph{comparable}.

\subsection{Preliminaries}

\medskip

To start with, assume that the nonelementary, geometrically finite group $G$ contains parabolic elements. As usual, let $P$ denote a complete set of inequivalent parabolic fixed points of $G$. Clearly, the orbit $G(P)$ of points in $P$ under $G$ is the set of all parabolic fixed points of $G$. To each $p \in P$ we associate a horoball $H_p$ and we write $H_{g(p)} $ for the image of $H_p$ under $g \in G$. It is well known that the horoballs $ H_p \; (p \in P)$ can be chosen so that their images under $G$ comprise a set of pairwise disjoint horoballs, see \cite{Bowditch}. By construction, any set $ \{ H_\xi : \xi \in G(P) \} $ chosen in this manner is $G$-invariant and is said to be a {\em standard set of horoballs} for $G$. Naturally, a horoball in a standard set is called a standard horoball.

A relatively simple argument shows that the top $\sigma_\xi$ of any standard horoball $H_\xi$ lies within a bounded hyperbolic distance (dependent only on $G$ and $P$) from the orbit of the origin under $G$ \cite[Lemma 2.2]{Strap}. In view of this, for each $p\in P$, there is a geometrically motivated set $\TT_p$ of representatives of the cosets $\{g G_p : g\in G\}$, such that for all $g\in \TT_p$, the orbit point $g(0)$ lies within a bounded distance (dependent only on $G$ and $P$) from the top of the standard horoball $H_{g(p)}$. Indeed, for each $\xi \in G(P)$ we may choose $g\in G$ so as to minimize $L_g$ subject to the restriction that $g(p) = \xi$, and then we can let $\TT_p$ be the collection of all group elements chosen in this manner. Let $R_g$ denote the Euclidean radius of $H_{g(p)}$. We remark that $R_g$ is only defined for $g \in \mathcal{ T}_p $ and that
\begin{equation} \label{comp} R_g \, \asymp \, 1-|g(0)| \, \asymp L_g \, . \end{equation}
This together with the fact that the horoballs in a standard set are disjoint implies the following extremely useful statement; see \cite[\S2.4]{jbgfg}.

\begin{lemma}[Disjointness] \label{disjointness}
Let $G$ be a nonelementary, geometrically finite Kleinian group containing parabolic elements and let $P$ denote a complete set of inequivalent parabolic fixed points of $G$. There is a constant $c_1> 0 $ depending only on $G$ and $P$ with the following property: for all $p,q \in P$ and $ g,h \in G$ such that $g(p) \neq h(q)$, we have

\[
| g(p) - h(q)| > \frac{c_1}{\sqrt{L_g L_h}} \, .
\]
In particular, fix $k > 1$ and suppose that $ k^{n} \le L_g, L_h < k^{n+1}$ for some $n \in \N$. Then
\[
B\big(g(p), c_2/L_g \big) \cap B\big(h(q), c_2/L_h \big) = \emptyset
\]
where $ c_2 := c_1/2 \sqrt{k} $.
\end{lemma}

\noindent For further details regarding the above notions and statements see any of the papers \cite{jbgfg,melian,Strap} and the references within.
\medskip

We now turn our attention to the situation where the geometrically finite group $G$ has no parabolic elements. Let $\{\eta,\eta'\}$ be the pair of fixed points of a hyperbolic element of $G$. Let $L$ be the axis of the corresponding hyperbolic element of $G$, or equivalently the bi-infinite hyperbolic geodesic connecting $\eta$ with $\eta'$. Let $G_{\eta\eta'}$ denote the stabiliser of $\eta$, which can easily be shown to be equal to the stabiliser of $\eta'$. Then there is a geometrically motivated set $\mathcal{ T}_{\eta \eta'} $ of coset representatives of $G/G_{\eta \eta'}$; chosen so that for all $ g \in \mathcal{ T}_{\eta \eta'}$, the orbit point $g(0)$ lies within a bounded hyperbolic distance from the summit $s_g$ of $g(L)$. Here $g(L)$ denotes the image of $L$ under $g$ and is equal to the bi-infinite geodesic connecting the hyperbolic fixed points $g(y)$ and $g(y')$. The summit $s_g$ is the point on $g(L)$ ``closest'' to the origin, or equivalently the midpoint of $g(L)$ when $g(L)$ is thought of as the arc of a Euclidean circle rather than as a bi-infinite hyperbolic geodesic. Now for $g \in \mathcal{ T}_{y y'}$ and $y\in \{\eta,\eta'\}$, let $H_{g(y)}$ be the horoball with base point at $g(y)$ and radius $R_g := 1 - |s_g| $. Then the top of $H_{g(y)}$ lies within a bounded hyperbolic distance of $g(0)$ and it follows that \eqref{comp} holds for all $g \in \mathcal{ T}_{y y'}$. The following statement is the analogue of Lemma \ref{disjointness} for convex cocompact groups.

\begin{lemma} \label{disjointnessH}
Let $G$ be a nonelementary, geometrically finite Kleinian group without parabolic elements and let $\{\eta,\eta'\}$ be the pair of fixed points of a hyperbolic element of $G$. There is a constant $c_1> 0 $ depending only on $G$ and $\eta,\eta'$ with the following property: for all $u,v \in\{\eta,\eta'\}$ and $g,h \in G$ such that $g(u) \neq h(v)$, we have
\[
| g(u) - h(v)| > \frac{c_1}{\max \{ L_g, L_h\} } \, .
\]
In particular, fix $k > 1$ and suppose that $ k^{n} \le L_g, L_h < k^{n+1}$ for some $n \in \N$. Then
\[
B\big(g(u), c_2/L_g \big) \cap B\big(h(v), c_2/L_h \big) = \emptyset
\]
where $ c_2 := c_1/2 k $.
\end{lemma}

%
%
%
%

%

This lemma was established by Patterson \cite[Theorem 7.2]{Paddyrs}. He dealt only with the Fuchsian case but his proof extends trivially to higher dimensions.

\subsection{Proof of Theorem \ref{singthm} \label{prrofthm1} }

We prove the theorem in the case where $G$ has parabolic elements. The proof in the case where $G$ is without parabolic elements is essentially identical, except that one appeals to Lemma \ref{disjointnessH} rather than Lemma \ref{disjointness}.

Fix $\xi \in \Lambda $. Trivially, if $\xi$ is a parabolic fixed point of $G$ then it is singular. To prove the opposite implication, assume that $\xi $ is singular. Fix $\ep > 0$ small, to be determined later. Then by definition, there exists $N_0$ such that for all $ N \ge N_0$ there exist $p = p_N \in P$, $g = g_N \in G$ so that
\begin{equation}\label{singhyperp}
| \xi - g(p) | \ < \
\frac{\ep}{\sqrt{L_g N}} \qquad {\rm and } \quad
L_g < N. \ \end{equation}
If $g_N(p_N) = g_{2N}(p_{2N})$ for all $N \ge N_0$, then since the right-hand side of the first inequality of \eqref{singhyperp} tends to zero as $ N \to \infty$, we have that $\xi = g_N(p_N)$ for all $N \geq N_0$. In other words, $\xi$ is a parabolic fixed point of $G$ and we are done. Thus, assume that $g_N(p_N) \neq g_{2N}(p_{2N})$ for some $N \ge N_0$. Write $g = g_N$, $p = p_N$, $h = g_{2N}$, and $q = p_{2N}$. Then
\begin{equation}\label{singhyperpp}
| \xi - h(q) | \ < \
\frac{\ep}{\sqrt{L_h \, 2N}} \qquad {\rm and } \quad
L_h < 2 N \ . \end{equation}
It then follows via \eqref{singhyperp}, \eqref{singhyperpp}, and the disjointness lemma (Lemma \ref{disjointness}) that there is a constant $c_1 > 0 $ depending only on $G$ and $P$ so that
\begin{eqnarray*}
\frac{c_1}{\sqrt{L_g L_h}} \, < \, | g(p) - h(q)| & = & | (\xi - h(q)) - (\xi - g(p))| \\
& < & \frac{\ep}{\sqrt{L_h \, 2N}} \, + \, \frac{\ep}{\sqrt{L_g N}} \\ [2ex]
& < & \frac{\ep}{\sqrt{2L_h L_g}} \, + \, \frac{\sqrt2\ep}{\sqrt{L_g L_h}} \\ [2ex]
& < & \frac{3\ep}{\sqrt{L_gL_h}} \, .
\end{eqnarray*}
The upshot is that we obtain a contradiction by setting $ \ep \le c_1/3$. This completes the proof.

\subsection{Proof of Theorem \ref{mainext} }

As before, we prove the theorem in the case where $G$ has parabolic elements. The proof in the case where $G$ is without parabolic elements is essentially identical, except that one appeals to Lemma \ref{disjointnessH} rather than Lemma \ref{disjointness}.

For all $\ep > 0$, we have
\begin{equation} \label{qw}
\psi(r) \, < \, \ep \, r^{-1} \qquad {\rm for \ sufficiently \ large \ } r.
\end{equation}
To see this note that since $\psi$ is decreasing, we have
\begin{equation}
\label{psimon}
\sum_{n/2 < r \leq n }
r^{\alpha -1 } \p(r)^{\alpha} \ \gg \ \sum_{n/2 < r \leq n }
n^{\alpha -1 } \p(n)^{\alpha} \ \asymp \
n^{\alpha } \p(n)^{\alpha} \,
\end{equation}
for every natural number $n$. In view of the convergent sum condition associated with \eqref{ineqmain2}, we have that the left-hand side of the above inequality tends to zero as $n \to \infty$. Hence
\[
n \, \p(n) \to 0 \quad {\ \rm as \ } \quad n \to \infty \,
\]
and \eqref{qw} follows. Also note that the convergent sum condition together with \eqref{psimon} implies that
\begin{equation} \label{qwe}
\sum_{n=1}^{\infty}
\big( 2^n \p(2^n) \big)^{\alpha} \ < \ \infty \, .
\end{equation}
Next, let
\[
W^*_{p}(\psi)
:=\big\{ \xi \in \Lambda: | \xi - g(p) | < \psi(L_g) \, \mbox{ \ for \ i.m.
$g \in \mathcal{ T}_p $} \big\}
\]
and observe that since $\psi$ is monotonic, by the definition of $\TT_p$ we have
\[
W_{p}(\psi) \; = \; W^*_{p}(\psi) \; \cup \; G(p) \, .
\]
The set $ G(p):= \{g(p) : g \in G \} $ is countable and so $ W_{p}(\psi) $ and $W^*_{p}(\psi)$ have the same $\mu$-measure; in particular

\[
\mu ( W_{p}(\psi)) = 0 \quad \Longleftrightarrow \quad \mu (W^*_{p}(\psi)) = 0 \, .
\]
Now for each $n \in \N $, let

\[
A_p(\psi, n) \ := \bigcup_{\substack{g \in \mathcal{ T}_p \, : \\ 2^{n} < \, L_g \leq 2^{n+1} }}
\!\! B\big(g(p),\p( L_g ) \big) \ .
\]
By definition,

\[
W^*_{p}(\psi) = \limsup_{n \to \infty} A_p(\psi, n) \,
\]
and by the Borel-Cantelli Lemma
\begin{equation} \label{the game}
\mu \big(W_{p}(\psi) \big) = \mu \big(W^*_{p}(\psi) \big) = 0 \qquad {\rm if } \qquad \sum_{n=1}^{\infty} \mu ( A_p(\psi, n)) < \infty \, .
\end{equation}
Thus, the name of the game is to show that the above sum converges.

\medskip

In view of \eqref{qw}, for $n $ sufficiently large we can assume that
\[
\psi(n) \, < \, \frac{c_2}{8 \, n} \, .
\]
Here $c_2$ is the absolute constant appearing in Lemma~\ref{disjointness} with $k=2$. It then follows from Lemma~\ref{disjointness} that for $n$ large enough, the union of balls associated with $A_p(\psi, n)$ is a disjoint union and so
\[
\mu \big( A_p(\psi, n) \big) \; = \; \sum_{\substack{g \in \mathcal{ T}_p \, : \\ 2^{n} < \, L_g \leq 2^{n+1} }}
\!\! \mu \Big( B\big(g(p),\p( L_g ) \big) \Big) \ .
\]
The measure $\mu$ is supported on $K$ and so the only balls that can potentially make a positive contribution to the above sum are those that intersect $K$. With this in mind, take such a ball $ B\big(g(p),\p( L_g ) \big)$ and choose a point
\[
\widetilde{g(p)} \in K \cap B\big(g(p),\p( L_g ) \big) \, .
\]
It is easily verified that
\[
B\Big(g(p),\p( L_g ) \Big) \subseteq B\Big(\widetilde{g(p)},2 \p( L_g ) \Big) \subseteq B\Big(\widetilde{g(p)}, \textstyle{\frac{c_2}{ 2 L_g}} \Big) \subseteq B\Big(g(p), \textstyle{\frac{c_2}{ L_g}} \Big) \, .
\]
Since $\mu$ is weakly absolutely $\alpha$-decaying, it follows that for $n$ sufficiently large
\begin{eqnarray} \label{finale}
\mu \big( A_p(\psi, n) \big) \; & \le & \; \sum_{\substack{g \in \mathcal{ T}_p \, : \\ 2^{n} < \, L_g \le
2^{n+1} }}
\!\! \mu \Big( B\big(\widetilde{g(p)}, \; 2 \, \p( L_g ) \big) \Big) \nonumber \\ [2ex]
& = & \; \sum_{\substack{g \in \mathcal{ T}_p \, : \\ 2^{n} < \, L_g \le
2^{n+1} }}
\!\! \mu \Big( B\Big(\widetilde{g(p)}, \; 2 \, \p( L_g ) \, \textstyle{\frac{2c_2 L_g}{2c_2 L_g}} \Big) \Big) \nonumber \\ [2ex]
& \le & \; \sum_{\substack{g \in \mathcal{ T}_p \, : \\ 2^{n} < \, L_g \le
2^{n+1} }}
\!\! C \, \big( 2 \, \p( L_g ) 2 L_g c_2^{-1} \big)^{\alpha} \; \mu \Big( B\Big(\widetilde{g(p)}, \; \, \textstyle{\frac{c_2}{2 \, L_g}} \Big) \Big) \nonumber \\ [2ex]
& \le & \; C \, \big( 8 \, c_2^{-1} \, \p( 2^n ) 2^n \big)^{\alpha} \; \sum_{\substack{g \in \mathcal{ T}_p \, : \\ 2^{n} < \, L_g \le
2^{n+1} }}
\!\! \mu \Big( B\Big(g(p), \; \, \textstyle{\frac{c_2}{ L_g}} \Big) \Big) \, .
\end{eqnarray}
The measure $\mu$ is a probability measure and by Lemma~\ref{disjointness} the balls associated with the above sum are disjoint. Hence
\[
\sum_{\substack{g \in \mathcal{ T}_p \, : \\ 2^{n} < \, L_g \le
2^{n+1} }}
\!\! \mu \Big( B\Big(g(p), \; \, \textstyle{\frac{c_2}{ L_g}} \Big) \Big) \ \le \ 1 \,
\]
which together with \eqref{qwe} and \eqref{finale} implies that
\[
\sum_{n=1}^{\infty} \mu \big( A_p(\psi, n) \big) \ \ll \
\sum_{n=1}^{\infty}
\big( 2^n \p(2^n) \big)^{\alpha} \ < \ \infty \, .
\]
\noindent In view of \eqref{the game}, this completes the proof of Theorem \ref{mainext}.
\medskip


\section{Proof of Theorem \ref{mainbad} \label{proofmainbad}}

Trivially, we have $K \cap \bad_y \subseteq K$. Hence, we immediately obtain the upper bound
\[
\dim (K \cap \bad_y) \le \dim K  = \delta \, .
\]
The usual strategy for proving the complementary lower bound inequality is to show that for all $s < \delta$, there exists a closed ``Cantor-like'' set $F_s \subset K \cap \bad_y$ which supports a probability measure $ \mu_s $ with the property that
\begin{equation}
\label{one_sided_power_law}
\mu_s \big(B(x, r) \big) \ \ll \ C \, r^s \hspace{7mm} \forall \ x \in K
\hspace{5mm} \forall \ r < r_0
\end{equation}
for some constant $C = C_s > 0$. According to the Mass Distribution Principle \cite[\S4.1]{falc}, \eqref{one_sided_power_law} implies that $\dim F_s \ge s$ and thus since $ F_s \subseteq K \cap \bad_y $, we have $ \dim (K \cap \bad_y) \ge s$. Since $s$ can be chosen arbitrarily close to $\delta$, we obtain the desired lower bound $\dim (K \cap \bad_y) \ge \delta$.

The question arises of how to construct the Cantor-like sets $\{F_s : s  < \delta\}$. One could use a ``hands-on'' approach in which there is a series of ``stages'' in the construction of $F_s$ and the proof explicitly describes how to construct each stage from the previous stage. However, these constructions often tend to follow the same sort of pattern: each stage $n \in \N$ corresponds to a set $S_n$ which can be written as the finite union of disjoint balls which are contained in $S_{n-1}$. There are certain ``obstacles'' to be avoided in the construction of the set $S_n$, but other than these obstacles the only relevant consideration is how many disjoint balls of a certain radius can fit into each ball of $S_{n-1}$. The uniformity in these kinds of constructions can be summarised by saying that many of them are instances of a \emph{single common} construction, whose applicability in any given situation can be tested by determining whether the relevant set is ``winning'' in the sense of Schmidt's game, an infinite game introduced by Schmidt in 1966.

Thus, instead of taking the ``hands-on'' approach, we will instead prove that the set $K\cap \bad_y$ is winning for Schmidt's game. It turns out to be most convenient to do this by combining a few results which are already known. Namely, a result of Mayeda $\&$  Merrill \cite{MayedaMerrill} states that the set $\bad_y$ is winning for a different game introduced by McMullen and known as the ``absolute game'', while the results of Broderick, Fishman, Kleinbock, Reich, and Weiss (hereafter abbreviated BFKRW) state that any set winning for the absolute game can be intersected with any sufficiently nice fractal (and in particular any Ahlfors $\delta$-regular fractal) to get a set winning for Schmidt's game (played on that fractal). This immediately implies that $K\cap \bad_y$ is winning for Schmidt's game (played on $K$) and according to a theorem of Fishman \cite{Fishman} this implies the existence of the family of sets $\{F_s : s < \delta \}$ described above, and in particular that $\dim(K\cap \bad_y) \geq \delta $. To make this paper more self-contained, in what follows we give the details behind this argument, as well as recalling the definition of Schmidt's game and the absolute game. Hopefully, our presentation will be accessible to a reader who is not an expert in playing these games and thus provide them with another (more powerful) approach towards proving  statements  such as Theorem \ref{mainbad}.

The games approach has some natural advantages over the hands-on approach. For one thing, the class of absolute winning sets is known to be invariant under quasi-symmetric transformations \cite[Theorem~2.2]{McMullen_absolute_winning}. Hence, if $G$ is of the first kind and $f:S^d \to S^d$ is a quasi-symmetric homeomorphism, then with appropriate modifications the above argument shows that $\dim(K\cap f(\bad_y)) = \delta$. For another thing, the class of absolute winning sets is closed under countable intersections (see \cite[Theorem~2]{Schmidt1} for the idea of the proof of this folklore result), so the above argument can also be modified to show that $\dim(K\cap \bigcap_y \bad_y) = \delta$, where the intersection is taken over all $y$ as in Theorem \ref{mainbad}. The countable intersection property also shows that $\bad_y$ can be intersected with an absolute winning set coming from some other mathematical setup (not necessarily related to Kleinian groups) and the intersection will still be large.

\subsection{Schmidt's game and Fishman's theorem}

We first define Schmidt's game and show that sets winning for Schmidt's game have large Hausdorff dimension.  The simplified account which we are about to present  is sufficient to bring out the main features of the games.

Let $K$ be a closed subset of $\R^d$. For any $0 < \alpha,\beta < 1$, \emph{Schmidt's $(\alpha,\beta)$-game} is an infinite game played by two players, Ayesha and Bhupen, who take turns choosing closed balls in $\R^d$ whose centers lie in $K$, with Bhupen moving first. The players must choose their moves so as to satisfy the relations
\[
B_1 \supset A_1 \supset B_2 \supset \cdots
\]
and
\[
\rho(A_k) = \alpha \rho(B_k) \text{ and } \rho(B_{k + 1}) = \beta \rho(A_k)  \ \text{ for } k\in\N,
\]
where $B_k$ and $A_k$ denote Bhupen's and Ayesha's $k$th moves, respectively, and where $\rho(B)$ denotes the radius of a ball $B$. Since the sets $B_1,B_2,\ldots$ form a nested sequence of nonempty closed sets whose diameters tend to zero, it follows from the completeness of $K$ that the intersection $\bigcap_k B_k$ is a singleton, say
\[
\bigcap_k B_k = \{x_\infty\} \, ,
\]
whose unique member $x_\infty$ lies in $K$. The point $x_\infty$ is called the \emph{outcome} of the game. A set $S \subseteq K$ is said to be \emph{$(\alpha,\beta)$-winning on $K$} if Ayesha has a strategy guaranteeing that the outcome lies in $S$, regardless of the way Bhupen chooses to play. It is said to be \emph{$\alpha$-winning on $K$} if it is $(\alpha,\beta)$-winning on $K$ for every $0 < \beta < 1$, and \emph{winning on $K$} if it is $\alpha$-winning on $K$ for some $0 < \alpha < 1$.  Informally, Bhupen  tries to stay away from the ``target'' set $S$ whilst Ayesha tries to land on $S$.

In view of the fact that $ S $ is a  subset of $K$, we trivially have that
\begin{equation}  \label{genlb}
\dim   S    \,  \le  \, \dim   K    \, .
\end{equation}
Thus, the main substance of the following statement is the complementary lower bound.

\begin{lemma}
\label{lemmafishman}
Let $K\subset \R^d$ be a closed Ahlfors $\delta$-regular set, and let $S \subset K$ be winning on $K$. Then
\[
\dim S  =  \delta.
\]
\end{lemma}

\noindent The above lemma was originally proved by Fishman \cite[Theorem 3.1]{Fishman} but shorter proofs have appeared in the literature since then, see for example  \cite[Proposition 2.5]{DFSU_sponges_BA} and  \cite[Lemma 5.8]{LukyanenkoVandehey}. The difference between these proofs and the one appearing below is that the one below emphasises the connection with the ``hands-on'' technique for producing Cantor sets with a certain property.

\begin{proof}[Proof of Lemma \ref{lemmafishman}]
Let $\alpha > 0$ be chosen so that $S$ is $\alpha$-winning, and fix $0 < \beta \leq 1/2$. Fix a winning strategy for Ayesha in the $(\alpha,\beta)$-game on $K$. We will construct a Cantor subset $F_\beta$ of $K$ via a sequence of stages. The stages will have the following properties:
\begin{itemize}
\item[1.] Each stage $n \in \N $ will correspond to a set $F_n \subset \R^d$ which is the union of finitely many disjoint balls centered in $K$. All of these balls will have the same radius $(\alpha\beta)^n \rho_0$, where $\rho_0 > 0$ is a constant, and are separated by distances of at least $(\alpha\beta)^n \rho_0$.
\item[2.] Each of the balls appearing at stage $n$ will be a subset of some ball appearing at stage $n-1$.
\item[3.] Each of the balls appearing at stage $n$ will have exactly $N(\beta) = \lfloor c_\beta \beta^{-\delta}\rfloor$ ``children'' at stage $n+1$ (i.e. balls appearing in the construction of $S_{n+1}$ which are subsets of the ball under consideration). Here $c_\beta> 0$ is a constant depending on $\beta$.
\item[4.] The intersection $F_\beta := \bigcap_{n=1}^\infty F_n$ will be a subset of $S$.
\end{itemize}
It is well known that for such a construction, the Hausdorff dimension of  $F_\beta$ is equal to $\frac{\log N(\beta)}{-\log(\alpha\beta)}$ (see e.g. \cite[Theorem 4]{Beardon}). It then follows that
\[
\dim S  \, \ge \, \dim F_\beta  = \frac{\log N(\beta)}{-\log(\alpha\beta)} = \frac{- \delta \log(\beta) + O(1)}{-\log(\alpha\beta)}  \ \tendsto{\beta\to 0}   \  \delta \, .
\]

\noindent Thus, in view of this and \eqref{genlb},  constructing a sequence of stages satisfying (1)--(4) will complete the proof of the lemma.  With this in mind, let   $F_0$ be any closed ball of radius $\rho_0$ centered in $K$. Now suppose that we have constructed the sets $F_0,\ldots,F_n$, and we want to construct the set $F_{n + 1}$. Fix a ball $B_n \subset F_n$. We need to specify what the ``children'' of $B_n$ are. By construction, there is a sequence of nested balls $B_0 \supset \cdots \supset B_n$ appearing in the construction so far. We will think of these balls as possible moves for Bhupen in Schmidt's $(\alpha,\beta)$-game on $K$. If Bhupen makes these moves, then Ayesha's winning strategy produces a response $A_n \subset B_n$ of radius $\alpha (\alpha\beta)^n \rho_0$. Now let $\rho_{n + 1} = \beta \rho(A_n) = (\alpha\beta)^{n + 1} \rho_0$, and let $\{B(x_i,\rho_{n + 1}):i = 1,\ldots,N\}$ be a maximal disjoint collection of balls in $A_n$ separated by distances of at least $\rho_{n + 1}$. Then the balls $\{B(x_i,4\rho_{n + 1}):i = 1,\ldots,N\}$ form a cover of $A_n$, so since $K$ is Ahlfors $\delta$-regular, we have $N \geq N(\beta) = \lfloor c_\beta \beta^{-\delta }\rfloor$. We consider the balls $\{B(x_i,\rho_{n + 1}): i = 1,\ldots,N(\beta)\}$ to be the children of $B_n$, since they could be used as legal moves for Bhupen in response to Ayesha's move $A_n$.

It is easy to check that (1)--(3) hold. To check that (4) holds, fix $x_\infty \in F_\beta$ and note that there exists an infinite nested sequence of balls $B_0 \supset B_1 \supset \cdots$ appearing in the construction whose point of intersection is $x_\infty$. This sequence corresponds to a possible strategy that Bhupen could use against Ayesha's winning strategy, so by the definition of a winning strategy, we have $x_\infty \in S$.
\end{proof}

\subsection{The absolute game and intersections with fractals}

We define the absolute game as introduced by McMullen in \cite{McMullen_absolute_winning} and show that any absolute winning set is winning (for Schmidt's game) on Ahlfors $\delta$-regular sets.

Let $\Lambda\subset \R^d$ be a closed set. For each $0 < \beta < 1$, the \emph{absolute $\beta$-game on $\Lambda$} is an infinite game played by two players, Ayesha and Bhupen, who take turns choosing balls in $\R^d$ with centers in $\Lambda$, with Bhupen moving first. The players must choose their moves so as to satisfy the relations
\begin{equation}
\label{absolute}
B_{k + 1} \subset B_k \setminus A_k
\end{equation}
and
\[
\rho(A_k) = \beta \rho(B_k) \text{ and } \rho(B_{k + 1}) = \beta \rho(A_k)   \ \text{ for } k\in\N,
\]
 where we use the same notation as when defining winning on $K$. Due to condition \eqref{absolute} we think of Ayesha as ``deleting'' her chosen ball $A_k$, whereas Bhupen is thought of as ``moving into'' his choice $B_k$. As before, the completeness of $\Lambda$ implies that the intersection $\bigcap_k B_k$ is a singleton, say $\bigcap_k B_k = \{x_\infty\}$, and the point $x_\infty\in \Lambda$ is called the \emph{outcome} of the game. A set $S \subseteq \Lambda$ is said to be \emph{absolute $\beta$-winning on $\Lambda$} if Ayesha has a strategy guaranteeing that the outcome lies in $S$, regardless of the way Bhupen chooses to play. It is said to be \emph{absolute winning on $\Lambda$} if it is $\beta$-winning for every $0 < \beta < 1$.

The following result regarding absolute winning sets is essentially a direct consequence of  \cite[Proposition~4.7]{BFKRW}.

\begin{lemma}
\label{lemmaBFKRW}
Let $S$ be an absolute winning set on a closed set $\Lambda\subset \R^d$, and let $K \subset \Lambda$ be a closed Ahlfors $\delta$-regular set. Then $K\cap S$ is winning on $K$.
\end{lemma}

\noindent Note that there is a technicality in relating the statement of \cite[Proposition~4.7]{BFKRW} to the above lemma.  Namely, the hypothesis of \cite[Proposition~4.7]{BFKRW} requires that $K$ is ``zero-dimensionally diffuse'' (cf. \cite[Definition~4.2]{BFKRW}) rather than Ahlfors $\delta$-regular. But it is easily verified that every Ahlfors $\delta$-regular set is zero-dimensionally diffuse.\footnote{In fact, it can be  verified that the class of zero-dimensionally diffuse sets is exactly equal to the class of uniformly perfect sets.} For the sake of clarity and self containment, we include a short proof of the lemma.

\begin{proof}[Proof of Lemma \ref{lemmaBFKRW}]
Since $K$ is Ahlfors $\delta$-regular, there exists $\alpha > 0$ with the following property: every ball $B(x,\rho)$ such that $x\in K$ and $\rho \leq 1$ contains two disjoint balls of radius $\alpha\rho$ centered on $K$ and separated by a distance of at least $2\alpha\rho$. Fix $0 < \beta < 1$. We claim that $K\cap S$ is winning for Schmidt's $(\alpha,\beta)$-game on $K$. Indeed, we know that there is a winning strategy for Ayesha in the absolute $\alpha\beta$-game on $\Lambda$: she responds to each of Bhupen's moves $B_k$ by ``deleting'' a ball $A_k^{(0)}$ of size $\alpha\beta\rho_k$, where $\rho_k$ is the radius of $B_k$. By the definition of $\alpha$, there exist two disjoint balls $A_k^{(1)},A_k^{(2)} \subset B_k$ of radius $\alpha\rho_k$ centered on $K$ and separated by a distance of at least $2\alpha\rho$. Since the diameter of $A_k^{(0)}$ is strictly less than $2\alpha\rho$, it intersects at most one of the balls $A_k^{(1)},A_k^{(2)}$. Ayesha's strategy for Schmidt's $(\alpha,\beta)$-game on $K$ is then simply to choose the other one of these balls, or to choose arbitrarily between the balls $A_k^{(1)},A_k^{(2)}$ if $A_k^{(0)}$ does not intersect either of them. This is a legal move within the setup of  Schmidt's game and thus Bhupen must respond by making a move of radius $\alpha\beta\rho_k$ centered in $K$. But since $K \subset \Lambda$, this move is centered in $\Lambda$ and thus corresponds to a legal next move in the absolute $\alpha\beta$-game on $\Lambda$. Thus both games can continue in the same manner, yielding the same outcome. Since Ayesha's strategy in the absolute game guaranteed that the outcome is in $S$, the same is true for her new strategy in Schmidt's game.
\end{proof}

\subsection{$\bad_y$ is absolute winning and the finale}

We first show that the set $\bad_y$ is absolute winning. This  together  with Lemmas~\ref{lemmafishman} and \ref{lemmaBFKRW}  will enable us  to easily deduce the conclusion of Theorem \ref{mainbad}.

\begin{lemma}
\label{lemmaMM}
Let $G$ be a nonelementary, geometrically finite group, and let $y$ be a parabolic fixed point of $G$, if one exists, and a hyperbolic fixed point otherwise. Then the set $\bad_y$ is absolute winning on $\Lambda$, the limit set  of $G$.
\end{lemma}

%
%

 The case of parabolic fixed points was proven in \cite{MayedaMerrill}.  The proof is not particularly long and so for the sake of clarity and self containment, we include a proof which also covers the case of hyperbolic fixed points.

\begin{proof}[Proof of Lemma \ref{lemmaMM}]

Fix $0 < \beta < 1$. We specify Ayesha's strategy for the absolute $\beta$-game by describing how she would react to any ball $B = B(x,\rho)$ that Bhupen could choose. Let $k = \beta^{-1}$ and let $c_3 = c_1/4k$, where $c_1 > 0$ is as in Lemma~\ref{disjointness} (in the case where $y$ is parabolic) or Lemma~\ref{disjointnessH} (in the case where $y$ is hyperbolic). Then for all $g,h\in G$ such that $g(y) \neq h(y)$ and $k^n \leq L_g,L_h < k^{n + 1}$, the distance between the balls $B(g(y),c_3/L_g)$ and $B(h(y),c_3/L_h)$ is at least $c_1/2k^{n + 1}$, since
\[
|g(y) - h(y)| - \frac{c_3}{L_g} - \frac{c_3}{L_h} \geq \frac{c_1}{k^{n+1}} - \frac{c_1/4}{k^{n+1}} - \frac{c_1/4}{k^{n+1}} = \frac{c_1/2}{k^{n+1}} \cdot
\]
Now let $n$ be chosen so that $c_1/2k^{n + 2} \leq 2\rho < c_1/2k^{n + 1}$. Then we have shown that at most one ball of the form $B(g(y),c_3/L_g)$, $k^n \leq L_g < k^{n + 1}$, intersects Bhupen's ball $B(x,\rho)$. Ayesha's strategy can now be given as follows: ``delete'' the ball $B(g(y),\beta\rho)$, where $g\in G$ is chosen so that $B(g(y),c_3/L_g)$ intersects $B(x,\rho)$, if possible, and arbitrarily otherwise.

To complete the proof, we must show that this strategy guarantees that the outcome will lie in $\bad_y$. Indeed, as usual let $x_\infty$ denote the outcome, and consider an element $g\in G$. Then we have $k^n \leq L_g < k^{n + 1}$ for some $n$, and if $L_g$ is sufficiently large then there must have occurred some stage in the game where the value of $n$ appearing in the previous paragraph is the same as the value of $n$ we are interested in. In this stage, Bhupen chose a ball $B = B(x,\rho)$ and Ayesha deleted a ball $B(h(y),\beta\rho)$. By the definition of the absolute game, the outcome $x_\infty$ must lie in the set $B(x,\rho) \setminus B(h(y),\beta\rho)$.

We must consider two cases $g(y) = h(y)$ and $g(y) \neq h(y)$. In the first case, since $x_\infty \notin B(h(y),\beta\rho)$ we have
\[
|x_\infty - g(y)| \geq \beta\rho \geq \frac{\beta c_1}{4 k^{n + 2}} \geq \frac{\beta c_1}{4 k^2 L_g} \; ,
\]
and in the second case, since $x_\infty \in B(x,\rho)$, we have $x_\infty \notin B(g(y),c_3/L_g)$ and thus
\[
|x_\infty - g(y)| \geq \frac{c_3}{L_g} \; .
\]
This completes the proof of the lemma.
\end{proof}

By combining Lemmas \ref{lemmafishman}, \ref{lemmaBFKRW} and \ref{lemmaMM}, we can prove Theorem \ref{mainbad}.

\begin{proof}[Proof of Theorem \ref{mainbad}]
By Lemma~\ref{lemmaMM}, the set $\bad_y$ is absolute winning on $\Lambda$. Thus by Lemma~\ref{lemmaBFKRW}, the set $K\cap \bad_y$ is winning on $K$. Finally, since $K$ is an Ahlfors $\delta$-regular set, Lemma~\ref{lemmafishman} implies that $\dim(K\cap \bad_y) = \delta$, as desired.
\end{proof}

\section{The Dream Theorem for Kleinian Groups \label{dreamG}}

We now turn our attention towards the problem of developing a ``manifold'' theory for Diophantine approximation on Kleinian groups beyond the extremal theory of \S\ref{KTTsec}. Namely, it would be desirable to establish the following analogue of Conjecture~1 for Kleinian groups.   To the best of our knowledge,  nothing to date has been formulated in this direction.\vspace*{2ex}

\begin{thmenv}{Conjecture 3 (The Dream Theorem for Kleinian Groups)}
Let $G$ be a nonelementary, geometrically finite Kleinian group of the first kind acting on $B^{d + 1}$ and let $y$ be a parabolic fixed point of $G$, if there are any, and a hyperbolic fixed point otherwise. Let $\cM \subseteq \Lambda = S^d$ be a nondegenerate manifold. Then
\begin{equation}\label{vb59}
| \cM\cap W_y(\psi) |_{\cM}
=\left\{
\begin{array}{ll}
0 & \textup{if }
\sum_{r=1}^\infty \; \psi\left(r\right)^{d} r^{d - 1}
<\infty\; ,\\ [3ex]
1 & \textup{if }
\sum_{r=1}^\infty \; \psi\left(r\right)^{d} r^{d - 1}
=\infty \; .
\end{array}
\right.
\end{equation}
\end{thmenv}

\vspace*{2ex}

The convergence/divergence criterion appearing in \eqref{vb59}  is derived from the statement of Theorem KT.  In fact, in view of  Theorem KT, Conjecture 3 can be thought of as asserting that a typical point on a nondegenerate manifold $\cM \subset S^d$ has ``the same Diophantine properties'' as a typical point on $S^d$.
In the classical setup of Conjecture~1, the case of nondegenerate analytic planar curves is the ``easiest'' to handle. The natural analogue of this case within the Kleinian group setup   is the case of nondegenerate analytic curves in $S^2$. We therefore also record the following weaker conjecture. In short, it corresponds to Conjecture 3 with $d=2$ and an analyticity assumption.

\vspace*{2ex}

\begin{thmenv}{Conjecture 3B}
Let $G$ be a nonelementary, geometrically finite Kleinian group of the first kind acting on $B^3$ and let $y$ be a parabolic fixed point of $G$, if there are any, and a hyperbolic fixed point otherwise. Let $\cC \subseteq \Lambda = S^2$ be a nondegenerate connected analytic curve. Then
\begin{equation}\label{3prime}
| \cC\cap W_y(\psi) |_{\cC}
=\left\{
\begin{array}{ll}
0 & \textup{if }
\sum_{r=1}^\infty \; \psi\left(r\right)^{2} \; r
<\infty\; ,\\ [3ex]
1 & \textup{if }
\sum_{r=1}^\infty \; \psi\left(r\right)^{2} \; r
=\infty \; .
\end{array}
\right.
\end{equation}
\end{thmenv}

\vspace*{2ex}

\noindent A connected analytic manifold is nondegenerate if and only if it is not contained in any hyperplane. Thus in the statement of Conjecture 3B, the phrase ``nondegenerate connected analytic curve'' could be replaced with the phrase ``connected analytic curve not contained in any hyperplane of $\R^3$'' without changing the meaning of the conjecture.

\vspace*{2ex}
The condition that the manifold is nondegenerate is a necessary assumption in both conjectures. It naturally excludes situations of the following type for which the conclusion of the conjectures is clearly false. Given a group $G$, let us write $ W_y(\psi, G) $ for $ W_y(\psi) $ to emphasise the fact that (by definition) the set of $\psi$-well approximable limit points depends on the group $G$ under consideration. Now suppose there exists a geometrically finite subgroup $H$ of $G$ with parabolic elements preserving the subspace $B^2\times\{0\} \subset B^3$ with limit set $ \Lambda(H)= S^1\times \{0\} \subset S^2$. Let $p$ be a parabolic fixed point of $H$, which is then also a parabolic fixed point of $G$. Now with Conjecture 3B in mind, let $ \cC = S^1 \times \{0\} $. Then $\cC$ is a connected analytic curve which is degenerate (since $\cC \subset \R^2\times\{0\}$). It follows directly from the definition that
\[
W_p(\psi,H) \subset \cC\cap W_p(\psi,G).
\]
In fact, if $\psi$ decays fast enough then the reverse inclusion also holds (so that the two sets are equal), but we will not prove this fact here. So we have
\[
|\cC\cap W_p(\psi,G)|_\cC \geq |W_p(\psi,H)|_\cC
\]
and thus by Theorem KT applied to $H$,
\[
|\cC\cap W_p(\psi,G)|_\cC = |W_p(\psi,H)|_\cC = 1  \  \textup{ if }
  \   \sum_{r=1}^\infty \; \psi(r) =\infty \; .
\]
If we let $\psi(r) = (r\log(r))^{-1}$, then this shows that $|\cC\cap W_p(\psi,G)|_\cC = 1$, while direct calculation shows that $\sum_{r=1}^\infty \; \psi^2(r) r \; <\infty$. This means that the conclusion of Conjecture~3B is not valid for the group $G$ and the curve $\cC$.

Although there are many well-known methods for constructing a geometrically finite group $G$ of the first kind which admits a subgroup $H$ as described above, we list one for concreteness: the integer Lorentz group $G = SO(3,1;\Z) = SO(3,1)\cap SL_4(\Z)$ can be viewed as acting on the space
\[
\Half^3 = \{x\in \R^4 : -x_0^2 + x_1^2 + x_2^2 + x_3^2 = -1, \; x_0 > 0\},
\]
which is the hyperboloid model of three-dimensional hyperbolic space. It is a geometrically finite group of the first kind. The subgroup $H = {\rm Stab}(G;\{x_3 = 0\}) = SO(2,1;\Z) \oplus \{1\}$ is also geometrically finite (with parabolic elements), and if we conjugate from the hyperboloid model to the unit ball model then it preserves the subspace $B^2\times\{0\}$ and has $S^1\times\{0\}$ as its limit set.

\vspace*{2ex}

%
%
%
\vspace*{2ex}

\begin{rem}
Note that for the coarser extremal theory described in  $\S\ref{KTTsec} $, all that is required is that the compact subset $K$ of the limit set $\Lambda$ supports a weakly absolutely $\alpha$-decaying  measure.  Indeed, if we take  $K$ to be a submanifold  $\cM$ of $S^d$ and $\mu$ to be the Lebesgue measure on $\cM$, then it is easily verified  that $\mu$ is weakly absolutely $\alpha$-decaying with $\alpha =k:= \dim \cM$. Thus in this scenario, Theorem \ref{mainext} implies that
\[
| \cM \cap W_{y}(\psi) |_{\cM} = 0\hspace{6mm} {\rm if \ } \hspace{6mm} \sum_{r=1}^{\infty}
\p(r)^{k} r^{k  -1 } \ < \ \infty \ .
\]
However, since $k< d$, this falls short of the desired convergence-case statement associated with \eqref{vb59}.
\end{rem}

\vspace*{2ex}

\begin{rem}
Although we could have stated Conjecture 3 without making the assumption that $G$ is of the first kind, we would still have to assume at least that $|\cM \cap \Lambda|_{\cM} > 0$, since otherwise  we would have $|\cM \cap W_y(\psi)|_{\cM} = 0$ regardless of what $\psi$ is. Thus,  if $G$ is of the second kind it is not clear whether the resulting conjecture would be non-vacuous, since many groups of the second kind have totally disconnected limit sets. Even if the limit set $\Lambda$ is not totally disconnected, it is not clear whether or not it can contain a nondegenerate manifold. This in fact leads to another problem which as far as we know is open.
\begin{question}
Does there exist a  geometrically finite group of the second kind $G$ acting on $B^{d + 1}$ such that for some nondegenerate manifold $\cM \subset S^d$, we have that $$|\cM \cap \Lambda|_{\cM} > 0  \, ?$$
\end{question}
\noindent It  is easy to come up with examples where the limit set contains a degenerate  manifold.  \end{rem}

\subsection{Counting orbit points close to manifolds}
At its core, Conjecture 3 is a claim regarding the distribution of orbit points $g(y) $ ``close'' to the manifold  $\cM$. The following discussion brings this out to the forefront. For ease of exposition we restrict our attention to the case $d=2$, i.e. Conjecture 3B. Given a point $ \xi \in \Lambda = S^2$ and a set
$A \subseteq S^2$, let
\[
\dist(\xi , A ) := \inf\{| \xi - a | \, : \, a \in A \} \, .
\]
Now fix $\xi \in \cC \cap W_y(\psi) $. Then by definition, there
exist infinitely many $g\in G$ such that
\[
\dist(g(y) , \cC ) \le | \xi - g(y) | < \psi(L_g) \, .
\]
This means that the orbit points of interest $g(y)$ must lie in the $\psi(L_g)$--neighbourhood of $\cC$. In particular, since $\psi$ is decreasing, for each integer $k \ge 2$, the points of interest $ g(y)$ with $ k^{n} < L_g \le k^{n+1} $ are contained in the $\psi(k^{n})$--neighbourhood of $\cC$. Let us denote this neighbourhood (as a subset of $S^2$) by $\Delta^{\cC}_y (n,\psi)$, and let $N_y^{\cC} (n,\psi)$ denote the set of points $g(y) $ with $ k^{n} < L_g \le k^{n+1} $ contained in $\Delta_y^{\cC}(n,\psi)$. In other words,
\begin{equation*}\label{vb1}
N_y^{\cC}(n,\psi) := \left\{ g(y) \, : \, g \in G \ {\rm \ with \ } \ k^{n} < L_g \le k^{n+1} \ {\rm \ and \ } \ \dist(g(y) , \cC ) \le \psi(k^{n}) \right\} \, .
\end{equation*}
Regarding the
$2$-dimensional Lebesgue measure $m_2$ of the neighbourhood $\Delta_y^{\cC}(n,\psi)$, we have that
\[
m_2\big( \Delta_y^{\cC}(n,\psi) \big) \ \asymp \ \psi(k^{n})
\,.
\]
Now let
\[
A_y(n):= \left\{ g(y) \, : \, g \in G \ {\rm \ with \ } \ k^{n} < L_g \le k^{n+1} \right\} \, .
\]
It is well known, see for example \cite[\S3]{melian} or \cite[\S4.1]{Strap}, that
\[
\# A_y(n) \ \asymp \ (k^n)^2 \, .
\]
Now, if we assume that the points in $ A_y(n) $ are ``fairly''
distributed within $S^2$, we would expect that
\begin{equation*}\label{vb1+}
\# \{ A_y(n) \cap \Delta_y^{\cC}(n,\psi) \} \ \asymp \ \# A_y(n) \times m_2\big( \Delta_y^{\cC}(n,\psi) \big) \, .
\end{equation*}

\noindent The upshot is the following heuristic estimate:
\begin{equation}\label{x1}
\# N_y^{\cC}(n,\psi) \ \asymp \ k^{2n} \, \psi(k^{n}) \,.
\end{equation}
Establishing this heuristic estimate would be a major first step towards Conjecture 3. Indeed, it is relatively straightforward to show that establishing the upper bound
\[
\# N_y^{\cC}(n,\psi) \ \ll \ k^{2n} \, \psi(k^{n})
\]
would already imply the convergence case of Conjecture 3. The corresponding lower bound is not by itself enough to prove the divergence case. Loosely speaking, we would also need to know that the points associated with the set $ N_y^{\cC}(n,\psi)$ are ``ubiquitous'' within $\Delta_y^{\cC}(n,\psi)$   -- see   \cite{memoirs,Beresnevich-Dickinson-Velani-07:MR2373145}.

\vspace*{2ex}

\subsection{The logarithm law for manifolds  \label{dreamGs} } For the sake of simplicity, let $G$ be a nonelementary, geometrically finite Kleinian group of the first kind acting on $B^{d+1}$. Suppose that $G$ has parabolic elements and as usual let $P$ denote a complete set of parabolic fixed points inequivalent under $G$. Then the associated hyperbolic manifold $\cH = B^{d+1} /G $ consists of a compact part $ X_0 $ with a
finite number of attachments:
\[
\cH = X_0 \; \; \cup \; \; \bigcup_{p \in P} Y_p
\]
where each $p \in P$ determines an exponentially ``thinning'' end $
Y_{p} $ -- usually referred to as a cuspidal end -- attached to
$X_0$. We shall write $0$ for the projection of the origin in $B^{d+1}$
to the quotient space $\cal H$. Let $S^d$ be the unit sphere of
the tangent space to $\cal H$ at $0$, and for every vector $v$ in
$S^d$ let $\gamma_v$ be the geodesic emanating from $0$ in the
direction $v$. Furthermore, for each $t \in \R^+$, let $ \gamma_v (t)
$ denote the point achieved after traveling time $t$ along
$\gamma_v$. Now fix $p\in P$. We define a function
\begin{eqnarray*}
\pen_{p}: \cH & \to & \R^+\\
x &\mapsto &\left\{\begin{array}{ll}
0 & x\notin Y_{p}\\[1ex]
\dist(x,0) & x\in Y_{p},
\end{array}\right.
\end{eqnarray*}
where $\dist$ is the induced metric on $\cal M$.  This is the penetration of $x$ into the cuspidal end $Y_{p}$. A relatively simple argument (see \cite{MP,SDS,slv3}) shows that there is a precise correspondence between the excursion pattern of a random geodesic into a cuspidal end $Y_p$ and the Diophantine properties of a random limit point of $G$ with respect to approximation by the base points of standard horoballs in the $G$-invariant collection $ \{H_{g(p)} \, : \, g \in \mathcal{ T}_p \} $. In particular, with reference to Remark \ref{like}, the statement regarding the normalised $d$-dimensional Lebesgue measure $m$ of the set $W_{y}(\psi_{\ep})$ has the following well-known dynamical corollary: for $m$-almost all directions $v \in S^d$,
\begin{equation} \label{sll}
\limsup_{t \to \infty} \frac{ \pen_{p}(\gamma_v ( t)) }{ \log t } \ = \ \frac{1}{d} \ \ .
\end{equation}
This is Sullivan's famous logarithm law for geodesics \cite{SDS}. The essence of Conjecture~3 is that Sullivan's law survives when we restrict the directions~$v \in S^d$ to appropriate subsets $K$ of $ S^d$. Specifically, Conjecture 3B implies the following ``manifold'' strengthening of Sullivan's logarithm law for geodesics. Let $\cC $ be a nondegenerate analytic curve on the unit (tangent) sphere $S^2$. Then \eqref{sll} holds (with $d=2$) for almost all (with respect to the Lebesgue measure on $\cC$) directions $v \in \cC \subseteq S^2$.

\vspace{7ex}

\noindent{\bf Acknowledgements.}  SV would like to thank the  Tata Institute of Fundamental Research (Mumbai), where the ideas of this work germinated, for its hospitality.  Also during his visit to TIFR, Anish Ghosh introduced him to ``dosa heaven'' in  Malabar Hill where some of the first proofs materialised  -- for this SV is forever in his debt. SV would also like to say a big THANK YOU to Paddy Patterson for the wonderful memories  as his post-doc  in the early nineties and for his inspiring vintage paper \cite{Paddyrs}  that much of this work is based on. Finally, many thanks to Iona and Ayesha for being caring, challenging and well rounded teenagers and the wonderful Bridget for staying with me even after fifty!

We would collectively like to thank the two referees for their detailed reports which have improved the accuracy of the paper.

\vspace*{3ex}


{\footnotesize

}

{\small

\vspace{5mm}

\noindent Victor V. Beresnevich: Department of Mathematics,
University of York,

\vspace{-2mm}

\noindent\phantom{Victor V. Beresnevich: }Heslington, York, YO10
5DD, England.


\noindent\phantom{Victor V. Beresnevich: }e-mail: vb8@york.ac.uk


\vspace{5mm}
\newpage
\noindent Anish Ghosh: School of Mathematics, TIFR

\vspace{-2mm}

\noindent\phantom{Anish Ghosh: }Homi Bhabha Road, Mumbai 400 005, India


\noindent\phantom{Anish Ghosh: }e-mail: ghosh@math.tifr.res.in


\vspace{5mm}

\noindent David S. Simmons: Department of Mathematics,
University of York,

\vspace{-2mm}

\noindent\phantom{David S. Simmons: }Heslington, York, YO10
5DD, England.


\noindent\phantom{David S. Simmons: }e-mail: David.Simmons@york.ac.uk

\noindent\phantom{David S. Simmons: }website: \url{https://sites.google.com/site/davidsimmonsmath/}


\vspace{5mm}

\noindent Sanju L. Velani: Department of Mathematics, University of York,

\vspace{-2mm}

\noindent\phantom{Sanju L. Velani: }Heslington, York, YO10 5DD, England.


\noindent\phantom{Sanju L. Velani: }e-mail: slv3@york.ac.uk

}

\end{document}